\documentclass[12pt, reqno]{amsart}
\usepackage{amsmath, amsfonts, amssymb, amsthm}

 \newtheorem{thm}{Theorem}[section]
 \newtheorem{Theorem}[thm]{Theorem}
 
\newtheorem{Definition}[thm]{Definition}
\newtheorem{Lemma}[thm]{Lemma}
 \newtheorem{lem}[thm]{Lemma}
 \newtheorem{prop}[thm]{Proposition}
 \theoremstyle{definition}
 
 \theoremstyle{remark}
 \newtheorem{rem}[thm]{Remark}
 \numberwithin{equation}{section}

\newcommand{\PP}{\textbf{PP}}
\newcommand{\Ff}{\mathbb F}

\def\Tr{\operatorname{Tr}}

\begin{document}

\title[Permutation Polynomials of Degree $6$ or $7$]
{Permutation Polynomials of Degree $6$ or $7$ over Finite Fields of
Characteristic $2$}

\author[Li, Chandler and Xiang]{Jiyou Li$^{*}$, David B. Chandler,  Qing Xiang$^\dagger$}

\address{Mathematics Department, Shanghai JiaoTong University, Shanghai 200240,
P.R. China} \email{lijiyou@sjtu.edu.cn}

\address{Department of Mathematical Sciences, University of
Delaware,\linebreak
  Newark, DE 19716, USA}
\email{davidbchandler@gmail.com}

 \address{Department of Mathematical Sciences, University of
Delaware,\linebreak
  Newark, DE 19716, USA} \email{xiang@math.udel.edu}
\hyphenation{New-ark}
\thanks{$^*$Research supported in part by Science and Technology
Commission of Shanghai Municipality (Grant No. 09XD1402500).}
\thanks{$^{\dagger}$Research supported in part by NSF Grant DMS 0701049, and by the
Overseas Cooperation Fund (grant 10928101) of China.}



\begin{abstract}
In \cite{D1}, Dickson listed all permutation polynomials up to
degree $5$ over an arbitrary finite field, and all permutation
polynomials of degree $6$ over finite fields of odd characteristic.
The classification of degree $6$ permutation polynomials over finite
fields of characteristic $2$ was left incomplete. In this paper we
complete the classification of permutation polynomials of degree 6
over finite fields of characteristic $2$. In addition, all
permutation polynomials of degree 7 over finite fields of
characteristic 2 are classified.
\end{abstract}

\maketitle

\section{Introduction}
Let ${\Ff}_q$ be a field of $q$ elements, where $q=p^t$, $p$ is a
prime. A polynomial $f(x)\in {\Ff}_q[x]$ is called a {\it
permutation polynomial} (\textbf{PP}) of $\Ff_q$ if the induced
function $f:c\mapsto f(c)$ from $\Ff_q$ to itself is a permutation
of $\Ff_q$. Permutation polynomials have been studied extensively in
the literature, see \cite{ln, gm1, gm2, G, gm3} for surveys of known
results on \textbf{PP}s.

In \cite{D1}, Dickson determined all permutation polynomials of
degree $6$ over finite fields of odd characteristic. The
classification of \textbf{PP}s of degree $6$ over finite fields of
characteristic $2$ is much more complicated. Concerning \textbf{PP}s
of general degree $n\geq 1$, Carlitz conjectured in 1966 that if $q$
is odd, then for each even positive integer $n$, there is a constant
$C_n$ such that when $q>C_n$, there do not exist \textbf{PP}s of
degree $n$ over ${\Ff}_q$. Carlitz's conjecture was resolved in the
affirmative by Fried, Guralnick and Saxl in \cite{F}. Wan \cite{W}
generalized the Carlitz conjecture to the following stronger
conjecture: If $q>n^4$ and $\gcd(n,q-1)>1$, then there are no
\textbf{PP}s of degree $n$ over ${\Ff}_q$. The Carlitz-Wan
conjecture was proved by Lenstra; an elementary version of Lenstra's
proof was given by Cohen and Fried in 1995 \cite{C}. For more
details we refer the reader to \cite{F,G,W}.



We are concerned with \textbf{PP}s of degree $6$ or $7$ over
$\Ff_{2^t}$. First let us consider the degree 6 case. When $t$ is
even, we have $\gcd(6,2^t-1)=3>1$. It then follows from the
Carlitz-Wan conjecture (Lenstra's theorem) that there are no
\textbf{PP}s of degree $6$ if $2^t>6^4$. Therefore we have an almost
complete classification of degree 6 \textbf{PP}s over $\Ff_{2^t}$
when $t$ is even. Indeed the case where $t$ is even and $t\geq 6$
was completely settled by Mertens in 1993, as reported by Mullen
\cite{G}. When $t$ is odd, the Carlitz-Wan conjecture (Lenstra's
theorem) does not apply since in this case we always have
$\gcd(6,2^t-1)=1$.

In this paper, we determine all permutation polynomials of degree
$6$ over ${\Ff}_{2^t}$. This result completes the table of
permutation polynomials of degree $\leq6$ given by Dickson in
\cite{D1}. We include the proof of the classification of
\textbf{PP}s over $\Ff_{2^t}$ when $t$ is even. The proof when $t$
is odd is more complicated, but similar to the $t$ even case. In
addition, we classify all permutation polynomials of degree 7 over
$\Ff_{2^t}$.

\vspace{0.1in}

\noindent {\bf Notation}. For $x\in\mathbb{R}$, let $(x)_0=1$
 and $(x)_k=x (x-1) \cdots (x-k+1)$
for $k\in $ $\mathbb{Z^+}=\{1,2,3,\ldots\} $. For $k\in
\mathbb{N}=\{0,1,2,\ldots\}$ the binomial coefficient ${x \choose
k}$ is defined by ${x \choose k}=\frac {(x)_k}{k!}$. When $x,k\in
\mathbb{N}=\{0,1,2,\ldots\}$ and $k>x$, we define ${x \choose k}=0$.

\section{Preliminaries}
In general it is very hard to determine whether a given polynomial
is a \textbf{PP}. The following well-known criterion is a useful
characterization of permutation polynomials over a finite field
${\Ff}_q$.

\begin{thm}[\textbf{Hermite and Dickson}]
Let ${\Ff}_q$ be a finite field of order $q$, where $q$ is a power
of a prime $p$. Then $f(x)\in {\Ff}_q[x]$ is a permutation
polynomial of ${\Ff}_q$ if and only if the following two conditions
hold:
\begin{enumerate}
\item $f(x)$ has exactly one root in ${\Ff}_q$;

\item for each integer $n$ with $1 \leq n \leq q-2$ and $n \not\equiv 0\
(\mathrm{mod}\ p)$, the reduction of $[f(x)]^n\ (\mathrm{mod}\
x^q-x)$ has degree $\leq q-2$.
\end{enumerate}
\end{thm}

Suppose $f(x)=\sum_{i=0}^n a_ix^i\in\Ff_{p^t}[x]$ is a polynomial of
degree $n$. Then we set $$\psi f(x)=\sum_{i=0}^n a_i^px^i.$$

\begin{lem}\label{shift}
If $f(x)$ is a {\em \textbf{PP}} of $\Ff_q$, then so are
$f_1(x)=af(bx+c)+d$ and $\psi f(x)$, for all $a,b \neq 0, c, d \in
{\Ff}_q.$
\end{lem}

By Lemma~\ref{shift}, when considering \textbf{PP}s of degree $n$
over $\Ff_q$ ($q=p^t$), it suffices to consider monic polynomials
$f(x)$ of degree $n$ satisfying the conditions that $f(0)=0$ and the
coefficient of $x^{n-1}$ is equal to $0$ if $p\nmid n$. Such a
\textbf{PP} will be called a \emph{normalized} \textbf{PP}. For
convenience, a monic polynomial in $\Ff_q[x]$ satisfying the above
two conditions will also be called a \emph{normalized polynomial}.

We define an equivalence relation on the set of polynomials over
$\Ff_{2^t}$.

\begin{Definition}
In this paper, two polynomials $f(x), g(x)\in \Ff_{2^t}[x]$ are said
to be \emph{equivalent} if either $g(x)=af(bx+c)+d$ or $g(x)=\psi
f(x)$, with $a,b \neq 0, c, d \in {\Ff}_{2^t}$.
\end{Definition}

For normalized polynomials over $\Ff_{2^t}$ of degrees $6$ or $7$,
if we are only concerned with their permutation behavior, then
Lemma~\ref{shift} also allows us to assume that the coefficient of
$x^5$ is either $0$ or $1$, by suitable choices of $a$ and $b$ in
the lemma (with $c=d=0$). Let $\mu $ be any element of $\Ff_{2^t}$
such that $\Tr(\mu)=1$, where $\Tr:\Ff_{2^t}\rightarrow \Ff_2$ is
the absolute trace. Suppose $f(x)=x^6+x^5+b x^4+c x^3+d x^2+e
x\in\Ff_{2^t}[x]$ is a polynomial of degree $6$.  Then the
coefficient of $x^4$ in $f(x+a)$ is either $0$ or $\mu$, where $a$
is a root of $x^2+x+b$ (if $\Tr(b)=0$) or of $x^2+x+b+\mu$ (if
$\Tr(b)=1$), respectively, in $\Ff_{2^t}$.

We will need the following classical result due to Lucas.

\begin{thm}[\textbf{Lucas}]\label{lucas}
Let $p$ be a prime, and $n, r$ be positive integers having the
following $p$-adic expansions:
$$n =a_0+a_1 p+a_2 p^2+\cdots +a_k p^k \ \ \ \  (0 \leq a_i \leq p-1 ,\forall \ 0 \leq i \leq k),$$
$$r =b_0+b_1 p+b_2 p^2+\cdots +b_k p^k \ \ \ \  (0 \leq b_i \leq p-1 ,\forall \ 0 \leq i \leq k). $$
Then
$$ {n \choose r } \equiv  \prod_{i=0}^k  {a_i \choose b_i }\;\;  ({\rm mod}\; p).$$
\end{thm}

We will also need to use multinomial coefficients, which we define
below. For all $n$, $r$, $k_1, \ldots , k_r$ in
$\mathbb{N}=\{0,1,2,\ldots\}$ with $k_1+\cdots +k_r=n$ and $r\geq
2$, we define the multinomial coefficient ${n \choose k_1, k_2,
\ldots, k_r }$ by
$${n \choose k_1, k_2, \ldots, k_r }=\frac{n!}{k_1!k_2!\cdots
k_r!}.$$

The following theorem is known as the multinomial theorem.
\begin{thm}
We have the following expansion:
$$(x_1+x_2+\cdots + x_r)^n =
\sum_{  {k_1+k_2+\cdots +k_r=n} \atop {k_1 \geq 0, k_2 \geq 0
,\ldots, k_r \geq 0  } } {n \choose k_1, k_2, \ldots, k_r }x_1^{k_1}
x_2^{k_2} \cdots x_r^{k_r}.$$
\end{thm}

The following is the multinomial analogue of the Lucas theorem.
\begin{prop}\label{prop2.5}
Let $p$ be a prime, and let $k_1,k_2,\ldots, k_r, n$ be nonnegative
integers having the following $p$-adic expansions:
\begin{eqnarray*}
n &=&a_0+a_1 p+a_2 p^2+\cdots +a_s p^s \ \ \ \  (0 \leq a_j \leq p-1, \forall\; 0\leq j\leq s),\\
k_i&=&b_{i0}+b_{i1} p+b_{i2} p^2+\cdots +b_{is} p^s \ \  (0 \leq
b_{ij} \leq p-1, \forall\; 1 \leq i\le r, 0\le j \leq s).
\end{eqnarray*}
Then
 $${n \choose k_1, k_2, \ldots, k_r }\equiv  {a_0  \choose b_{10},b_{20}, \ldots,
b_{r0} }
 \cdots  {a_s  \choose b_{1s},b_{2s}, \ldots, b_{rs} }\; ({\rm mod}\; p).$$
In particular,
$${n \choose k_1, k_2, \ldots, k_r }  \not\equiv 0\; ({\rm mod}\; p)
\iff\sum _{i=1}^{r}
 b_{ij} = a_j , \  \forall \  0 \leq j \leq s .$$
\end{prop}

\section{Permutation polynomials of degree $6$ over finite fields of characteristic $2$}\label{degree6}

Our aim in this section is to classify all permutation polynomials
of degree 6 of $\Ff_{2^t}$. First we note that in \cite{D1}, Dickson
already obtained some restrictions on the coefficients of these
polynomials.

\begin{thm}[\textbf{Dickson}]\label{thm4.1}
Let $$f(x)=x^6+a x^5+b x^4+ c x^3 +d x^2 +e x\in\Ff_{2^t}[x]$$ be a
{\em \textbf{PP}} of ${\Ff}_{2^{t}}$ such that $f(x)\neq x^6$ and
when $t=5$, $f(x)\neq x^6$, or $x^6+ax^5+a^4x^2$. Then
\begin{enumerate}
\item when $t$ is even, we have $c=a^3\neq 0$.

\item when $t$ is odd, we have $a \neq 0$, $c \neq
0$ and $c \neq a^3 $.
\end{enumerate}
\end{thm}

Based on this result we obtain the main result of this section.

\begin{Theorem}\label{main}
Let $\alpha$ and $\beta$ be roots of $x^3+x+1\in\Ff_2[x]$ and
$x^4+x+1\in\Ff_2[x]$, respectively, in some extension fields of
$\Ff_2$. The following are permutation polynomials of degree $6$
over $\Ff_{2^t}$:
$$\begin{array}{cl}
x^6&\mathrm{whenever} \ t\ \mathrm{is\ odd}\\
x^6+x^5+x^3+x^2+x& t=3\\
x^6+x^5+x^3+\alpha x^2+\alpha x& t=3\\
x^6+x^5+\alpha x^3 & t=3\\
x^6+x^5+x^4+x^3+x^2 & t=3\\
x^6+x^5+x^4+x^3+x & t=3\\
x^6+x^5+x^4+\alpha^3 x^3+\alpha^4 x^2+\alpha^6 x & t=3\\
x^6+x^5+x^4 & t=3\\
x^6+x^3+x^2 & t=3\\
x^6+x^5+x^3+\beta^3 x^2+\beta^5 x & t=4\\
x^6+x^5+\beta^3 x^4+x^3+\beta x^2+\beta^6 x & t=4\\
x^6+x^5+\beta^3 x^4+x^3+\beta^8 x^2+\beta^{13} x & t=4\\
x^6+x^5+x^2 & t=5.
\end{array}$$

All other {\em \PP}s of degree $6$ over $\Ff_{2^t}$ are equivalent
to one of the above.
\end{Theorem}

We will prove Theorem~\ref{main} in two steps. First we deal with
the $t$ even case, which was previously handled by Mertens, as
reported by Mullen in \cite{G}. We begin with a lemma.

\begin{lem} \label{lemm3.1}
Let $$f(x)=x^6+x^5+b x^4+ x^3 +d x^2 +e x$$ be a normalized
polynomial in $\Ff_{2^t}[x]$, and let $[x^{k}](f(x))^{n}$ denote the
coefficient of $x^k$ in the expansion of   $(f(x))^n$ (mod $x^q-x$),
where $q=2^t$. If $t$ is even and $m \geq 42$, where $2^t=6m+4$,
then in ${\Ff}_q$ we have:
\begin{eqnarray}
[x^{6m+3}]f(x)^{m+5}&=&(b^8+1+e^4)(1+b^2+e)\label{even1} \\
&+&(b^8+b^4+d^4)(e^2+d^2+e),
\nonumber \\
{[x^{6m+3}]}f(x)^{m+13} &=&(b^{32}+b^{16}+d^{16})(1+b^2+e^2+d^2).
\label{even2}
\end{eqnarray}
\end{lem}

\begin{proof}
The highest power of $x$ when we expand $f(x)^{m+13}$ is $6(m+13)$.
Since we want to find the coefficient of $x^{q-1}$ in the expansion
of $f(x)^{m+13}$ (mod $x^q-x$), the terms we are interested in are
of the form $x^{i(q-1)}$, $i\geq 1$. Thus if we need to consider the
terms $x^{i(q-1)}$, $i\ge 2$, it must be that
$$6(m+13)\ge 2(6m+3),$$
which is equivalent to $m\le 12$. Since we have assumed that $m\ge
42$, we do not need to consider the terms $x^{i(q-1)}$, $i\ge 2$,
when try to find the coefficient of $x^{q-1}$ in the expansion of
$f(x)^{m+13}$ (mod $x^q-x$). The same comment holds true when we try
to compute the coefficient of $x^{q-1}$ in the expansion of
$f(x)^{m+5}$ (mod $x^q-x$).

By the above comment and the multinomial theorem, the coefficient of
$x^{q-1}$ in the expansion of $f(x)^{m+5}$ (mod $x^q-x$) is equal to
$$
\sum_{
{i_1+i_2+\cdots +i_6=m+5} \atop {6i_1+5i_2+4i_3+3i_4+2i_5+i_6=6m+3}
} {m+5 \choose i_1,i_2, \ldots, i_6 }b^{i_3}d^{i_5}e^{i_6},$$ where
the multinomial coefficient is viewed modulo 2. We can easily find
all solutions to the system of equations:
 \begin{eqnarray*}
 i_1+i_2+i_3+i_4+i_5+i_6&=&m+5;\\
 6i_1+5i_2+4i_3+3i_4+2i_5+i_6&=&6m+3;\\
 i_1,i_2,\ldots ,i_6&\geq& 0
  \end{eqnarray*}
for which the multinomial coefficient ${m+5 \choose i_1,i_2, \ldots,
i_6 }$ is 1 modulo 2. We give some details below.

The above system of equations is equivalent to
 \begin{eqnarray*}
 i_1+i_2+i_3+i_4+i_5+i_6&=&m+5;\\
 i_2+2i_3+3i_4+4i_5+5i_6&=&27;\\
 i_1,i_2,\ldots ,i_6&\geq& 0.
  \end{eqnarray*}
Note that $m$ has the following $2$-adic expansion:
\begin{equation}
m=2^1+2^3+2^5 +\cdots+2^{t-3}.
\end{equation}
and thus
\begin{equation}
\label{e-1}
 m+5 =1+2^1+2^2+2^3+2^5+\cdots
 +2^{t-3}.
\end{equation}

Now, in view of Lucas' Theorem and Proposition~\ref{prop2.5}, the
multinomial coefficient ${m+5 \choose i_1,i_2, \ldots, i_6 }$
vanishes modulo 2 whenever any two of $i_1,\cdots,i_6$ have a $1$ in
the same digit of the $2$-adic expansion. For instance, the solution
 $(i_1,i_2,i_3,i_4,i_5,i_6)=(m-10,8,4,1,2,0)$
 gives ${m+5 \choose m-10,8,4,1,2,0}b^4d^2=b^4d^2$ for any $m\geq42$, since
 \begin{eqnarray*}
m+5&=&1+1\cdot 2^1+1 \cdot2^2+ 1 \cdot 2^3+0 \cdot2^4+1 \cdot2^5+
\cdots +1\cdot2^{t-3}\\
 m-10&=&0+0\cdot 2^1+0 \cdot2^2+ 0 \cdot 2^3+0 \cdot2^4+1 \cdot2^5+
\cdots +1\cdot2^{t-3} \\
 8&=& 0+0\cdot 2^1+0 \cdot2^2+ 1 \cdot 2^3+0 \cdot2^4+0 \cdot2^5+
\cdots +0\cdot2^{t-3} \\
4&=& 0+0\cdot 2^1+1 \cdot2^2+ 0 \cdot 2^3+0 \cdot2^4+0 \cdot2^5+
\cdots +0\cdot2^{t-3}  \\
 2&=&0+1\cdot 2^1+0 \cdot2^2+ 0 \cdot 2^3+0 \cdot2^4+0 \cdot2^5+
\cdots +0\cdot2^{t-3}\\
1&=&1+0\cdot 2^1+0 \cdot2^2+ 0 \cdot 2^3+0 \cdot2^4+0 \cdot2^5+
\cdots +0\cdot2^{t-3}.
  \end{eqnarray*}
and thus there are no carries in the sum $(m-10)+8+4+1+2+0=m+5$ for
any $m\geq42$. The other computations are similar. We omit the
details.

Similarly we can find the coefficient of $x^{q-1}$ in the expansion
of $f(x)^{m+13}$ (mod $x^q-x$). We leave the details to the reader.

The proof of the lemma is now complete.
\end{proof}

\begin{thm}\label{thm4.2}
There are no permutation polynomials of degree $6$ over
${\Ff}_{2^{t}}$ when $t>4$ is even.
\end{thm}
\begin{proof}
The cases where $t=4$, or 6 are easily checked, for example, by a
computer. Thus we assume that $t\ge 8$. Write $2^t=6m+4$. Then $m\ge
42$.

Assume to the contrary that $f(x)=x^6+ax^5+b x^4+ cx^3 +dx^2 +e
x\in\Ff_{2^t}[x]$ is a permutation polynomial of $\Ff_{2^t}$. In
view of Lemma~\ref{shift} and Theorem~\ref{thm4.1}, we may assume
that $f(x)=x^6+ x^5+b x^4+ x^3 +dx^2 +e x$ (i.e., $a=1$ and
$c=a^3=1$). By the Hermite-Dickson criterion, the coefficient
$[x^{6m+3}]f(x)^{m+13}$ must be zero. Thus the expression on the
right hand side of (\ref{even2}) is equal to zero. We consider two
cases.

\vspace{1 ex} \noindent\textbf{Case l.}  ${d}^2+{e}^2+{b}^2+1=0$. In
this case, we have $e=1+b+d$.  Then
$$f(x)=x^6+x^5+ b x^4 +x^3+d x^2+(1+b+d) x.$$
The above $f(x)$ has at least two roots, $0$ and $1$, in
$\Ff_{2^t}$. So $f(x)$ can not be a permutation polynomial of
$\Ff_{2^t}$, a contradiction.

\vspace{1 ex}

\noindent \textbf{Case 2.} ${d}^{8}+{b}^{16}+{b}^{8}=0$. In this
case, we have $d=b^2+b.$ Substituting $d$ by $b+b^2$ in
(\ref{even1}) and by the Hermite-Dickson criterion, we get
$$(b^8 +1+e^4)(1+b^2+e)=0.$$
It follows that $e=1+b^2$, and we have
$$f(x)=x^6+x^5+bx^4+x^3+(b^2+b)x^2+(b^2+1)x.$$
Once again, $f(x)$ has at least the two roots, $0$ and $1$, in
$\Ff_{2^t}$, a contradiction.

The proof of the theorem is complete.
\end{proof}

To classify permutation polynomials of degree $6$ over $\Ff_{2^t}$,
where $t$ is odd, we need more lemmas. The following lemma is well
known, see \cite[p.~56]{ln}.

\begin{lem}\label{lemm3.4}
The quadratic equation $x^2+x+b=0$, $b\in {\Ff}_{2^t}$, has a
solution in ${\Ff}_{2^t}$ if and only if  $\Tr(b)=0$.
 \end{lem}

\begin{Lemma}\label{oneroot}
Let $c\in{\Ff}_q\backslash \{0,1\}$, where $q$ is an odd power of
$2$.  Then the quintic polynomial
$$g(x)=x^5+cx^2+x+c^2+c$$
has exactly one root in $\Ff_q$.
\end{Lemma}
\begin{proof}
Suppose to the contrary that $g$ has two roots, $\alpha_1,
\alpha_2\in\Ff_q$. Set $A:=\alpha_1+\alpha_2$, and
$u:=\alpha_1\alpha_2$.  Then we have the factorization
$$g(x)=(x^2+Ax+u)(x^3+Ax^2+Bx+D),$$
where $A,B,D,u\in\Ff_q$. Setting the coefficients on the left and
right equal, we have:
\begin{eqnarray*}
B&=&A^2+u;\\
D&=&AB+uA+c\\
&=&A^3+c;\\
AD+uB&=&1\quad\mathrm{or}\\
Ac&=&A^4+uA^2+u^2+1;\\
uD&=&c^2+c\quad\mathrm{or}\\
uA^5+uA(Ac)&=&(Ac)^2+A(Ac).
\end{eqnarray*}
In the last equation, substituting $Ac$ by $A^4+uA^2+u^2+1$, we have
\begin{equation}\label{roots}\nonumber
u^4+Au^3+(A^4+A^3+A)u^2+(A^3+A)u+A^8+A^5+A+1=0.
\end{equation}
Making the substitution $u=w+A+1$, we get
\begin{equation}\label{deg4inw}
w^4+Aw^3+(A^4+A^3+A^2)w^2+A^8+A^6=0.
\end{equation}
Now we note that if $w=0$, then either $A=1$ or $A=0$.  If $A=0$, then
$u=\alpha_1=\alpha_2=1$, and substituting into $g(x)$ gives $c=0$,
which we do not allow.  If $A=1$, then $u=w=0$, and
$\{\alpha_1,\alpha_2\}=\{0,1\}$, which again we do not allow.  Thus
we can divide both sides of (\ref{deg4inw}) by $w^4$ to get
$$1+(A/w+A^2/w^2)+(A^4/w^2+A^8/w^4)+(A^3/w^2+A^6/w^4)=0.$$  Since $\Tr(1)=1$ and the
trace of each term in parentheses is zero, taking trace of both
sides of the last equation, we get $1=0$, which is absurd.
Therefore, $g(x)$ has at most one root in $\Ff_q$.

Now suppose $g$ has no roots in $\Ff_q$.  Then it is either
irreducible over $\Ff_q$, or it factors into irreducible second and
third degree polynomials over $\Ff_q$. In either case, there are at
least three roots lying in an extension field whose order is an odd
power of $2$, a contradiction. The proof is now complete.
\end{proof}

\begin{lem}\label{lemm3.2}
Let $$f(x)=x^6+x^5+b x^4+ c x^3 +d x^2 +e x$$ be a normalized
polynomial in $\Ff_{2^t}[x]$, and let $[x^{k}](f(x))^{n}$ denote the
coefficient of $x^k$ in the expansion of   $(f(x))^n$ (mod $x^q-x$),
where $q=2^t$. If $t$ is odd and $m \geq 85$, where $2^t=6m+2$, then
in ${\Ff}_{2^t}$ we have:
\begin{eqnarray*}
E_1&=&[x^{6m+1}]f(x)^{m+2}\\
&=&b^4(1+c)+(c^2+b^2c+e)+(e^2+cd^2+c^2e);\\
E_2&=&[x^{6m+1}]f(x)^{m+6}\\
&=& (b^{16}+b^8+d^8)(1+c)+(1+c^8)(e^2+cd^2+c^2e);\\
E_3&=&[x^{6m+1}]f(x)^{m+40}\\
&=&[(c^{64}+b^{16}c^{64}+
b^{64}c^{32}+b^{112}+b^{64}d^{16}
+e^{32}+b^{16}d^{32}+c^{32}d^{16}\\
&&+d^{16}e^{32})(1+c^4)\\
&+&(c^{64}+b^{96}+b^{64}c^{16}+b^{64}e^{16}+d^{32}
+c^{48}+b^{32}e^{16}+c^{16}e^{32}\\
&&+d^{32}e^{16})
 (c^8+b^8c^4+e^4)\\
&+&(c^{64}+b^{96}+b^{80}+b^{64}d^{16}
+d^{32}+b^{16}c^{32}+b^{32}d^{16}+b^{16}e^{32}+d^{48})\\
&&\cdot\ (e^8+c^4d^8+c^8e^4)\\
&+&(b^{64}+b^{64}c^{16}+c^{32}+b^{32}c^{16}+e^{16}+e^{32}
+d^{32}c^{16}+c^{32}e^{16})e^{12}]c
\end{eqnarray*}
\end{lem}

Similarly to the proof of Lemma~\ref{lemm3.1}, the coefficients
$[x^{q-1}]f(x)^{m+i}$, for various values of $i$, can be obtained by
hand. One can also use a computer to obtain these coefficients
easily. The following are two Maple procedures that we used for this
purpose. To use it, one sets the degree, ``deg,'' as well as ``m''
and ``r,'' where $q=\mathrm{deg}*\mathrm{m}+\mathrm{r}$, but the
value of ``m'' is reduced modulo a high enough fixed power of 2.
Then typing \verb{hermite(i){ computes $[x^{q-1}]f(x)^{m+i}$. For
polynomials of degree higher than $11$, the procedures need to be
modified slightly.

\noindent\begin{verbatim}m:=85;r:=2;deg:=6;
nextstage:=proc(ex,monoin,stage,tsum)
  local incr,tempsum,ind,monoout,exout,digit;
  global poly,deg,A,tot;
  exout:=iquo(ex,2,'digit');incr:=2^stage;
  if tsum+2*incr <= tot then
    nextstage(exout,monoin,stage+1,tsum);
    end if;
  tempsum:=tsum;
  if digit = 1 then
    for ind to deg-1 do
      tempsum:=tempsum+incr;
      if tempsum > tot then  break; end if;
      monoout:=monoin*A[ind]^incr;
      if tempsum=tot then poly:=poly+monoout;break;end if;
      nextstage(exout,monoout,stage+1,tempsum);
      end do;
    end if;
  end proc;

hermite:=proc(u)
  global m,r,deg, A,a,b,c,d,e,f,g,h,i,j,poly,tot;
  local exout;
  description "find [q-1]f^{m+u}";
     a:='a';b:='b';c:='c';d:='d';e:='e';
     f:='f';g:='g';h:='h';i:='i';j:='j';
     A:=array(1..10,[a,b,c,d,e,f,g,h,i,j]);poly:=0;
  tot:=deg*u-r+1;exout:=m+u; nextstage(exout,1,0,0); RETURN(poly);
end proc; hermite(2); hermite(6); hermite(40); \end{verbatim}

We are now ready to give the proof of Theorem~\ref{main} in the case
where $t$ is odd. We state the result separately as a theorem.

\begin{thm}\label{thm4.3}
Let $f(x)$ be a permutation polynomial of degree $6$ over
${\Ff}_{2^{t}}$, where $t$ is odd. Then
\begin{enumerate}
\item when $t=3$, $f(x)$ is equivalent to one of the degree 6
polynomials listed in the statement of Theorem~\ref{main};

\item when $t=5$, $f(x)$ is equivalent to either $x^6$ or $x^6+ x^5+
x^2$;

\item when $t>5$, $f(x)$ is equivalent to $x^6$.
\end{enumerate}
\end{thm}

\begin{proof} For odd $t$, we have $\gcd(6,2^t-1)=1$, thus $x^6$ is
a permutation polynomial of $\Ff_{2^t}$. Again, the cases where
$t<9$ are easily checked by a computer. From now on, we assume that
$t\geq 9$. We will prove that $f(x)=x^6+ax^5+bx^4+cx^3+dx^2+ex\in
\Ff_{2^t}[x]$ ($f(x)\neq x^6$) can not be a permutation polynomial
of $\Ff_{2^t}$ when $t\geq 9$.

By way of contradiction, assume that
$f(x)=x^6+ax^5+bx^4+cx^3+dx^2+ex\in \Ff_{2^t}[x]$ ($f(x)\neq x^6$)
is a \PP. In view of Part (2) of Theorem~\ref{thm4.1} and
Lemma~\ref{shift}, we may assume without loss of generality that
$$f(x)=x^6+x^5+b x^4+ cx^3 +dx^2 +e x,$$
where $c\neq 0$ or 1. By Lemma \ref{lemm3.2} and the Hermite-Dickson
criterion, we must have $E_1=0$, $E_2=0$, and $E_3=0$.

Using the equation $E_1=0$ to eliminate the variable $d$ from $E_2$,
we have
\begin{eqnarray*}
E_5:&=&c^4E_2\\
&=&(c+1)e^8+(c+1)^9e^4+c^4(c+1)^8e\\
&+&(c+1)(b^{16}+c^8+c^4b^4+c^{12}b^4)+c^5(c+1)^8(b^2+c)\\
&=&0.
\end{eqnarray*}
For compact expression we introduce
$$\gamma:=c(c+1)(c^2+b^2)+(c+b^2)^2.$$
Then we have the factorization
$$E_5=(c+1)\cdot E_4\cdot E_6, $$ where
 \begin{eqnarray*}
E_4&=&e^2 + (c + 1)e + c^4 + c^3 + (b^2 + 1)c^2 + b^2c + b^4\\
&=& e^2 + (c+1)e + \gamma
\end{eqnarray*}
and
\begin{eqnarray*}
E_6&=&e^6+(c+1)e^5+(\gamma+c^2+1)e^4+(c+1)^3e^3\\
&+&
(\gamma^2+(c^2+1)\gamma+c^8+c^4)e^2 + ((c+1)\gamma^2+c^4(c+1)^5)e
\\&+&
(\gamma^3+(c^8+c^4)\gamma+c^4(c+1)^6).
       \end{eqnarray*}
Since  $c\neq 1$, we must have $E_4=0$ or $E_6=0$.

Similarly, using $E_1=0$ to eliminate the variable $d$ from $E_3$,
we get an expression $E_8:=c^{23}E_3$ which can be factored

$$E_8=c(c+1)^4\cdot{E_4}^4\cdot{E_7}^4,$$
where
\begin{eqnarray*}
E_7&=&(c + 1)e^{11} + (c^4 + c^3 + b^2c^2 + b^2c + (b^4 + 1))e^{10}
+
(c^4 + b^8)e^8\\
& +& (c^7 + c^6 + b^4c^5 + b^4c^4 + b^4c^3 + b^4c^2 + c
+1)e^7\\
&+& (c^9 + (b^4 + b^2)c^8 + (b^4 + b^2)c^7 + (b^6 + b^4)c^6 +
(b^6 + b^4)c^5\\
& +& (b^6 + b^4)c^4 + (b^6 + 1)c^3 + (b^{12} + b^8 + b^4
+b^2)c^2
 + b^2c \\
&+& (b^4 + 1))e^6\\
& +& ((b^8 + 1)c^5 + (b^8 + 1)c^4 +
  (b^{12} + b^4)c^3 + (b^{12} + b^4)
        c^2)e^5\\
& +& (c^{10} + b^4c^8 + (b^8 +
        1)c^7 + (b^{12} + b^{10} + b^8 + b^2)
        c^6\\
&+& (b^{12} + b^{10} + b^4 + b^2)c^5
 +
        (b^{14} + b^6 + b^4 + 1)c^4\\
& +& (b^{14} + b^6)c^3 + (b^{16} + b^8 + b^4)c^2
+
        b^8)e^4\\
& +& (c^9 + c^8 + b^{16}c + b^{16})e^3 \\
&+& (c^{14} + c^{12} + c^{11} +
(b^4
        + b^2)c^{10}+ b^2c^9 + (b^4 + 1)        c^8 + b^{16}
        c^6\\
& +& b^{16}
        c^4
 + b^{16}c^3 +
        (b^{20} + b^{18})
        c^2 + b^{18}
        c + (b^{20} + b^{16}))e^2\\
& +& (c^{15} + c^{14} +
        b^4c^{11} + b^4
        c^{10} + b^{16}c^7 + b^{16}
        c^6 + b^{20}c^3 + b^{20}
        c^2)e\\
& +&
        c^{18} + c^{17} + b^2c^{16} + b^2c^{15}
        + c^{14} + b^4c^{13}+ (b^6 + 1)c^{12} +
        b^6c^{11}\\
& +& (b^{16} + b^8 + b^4)
        c^{10} + b^{16}c^9+ (b^{18} + b^8)
        c^8 +
        b^{18}c^7\\
& +& b^{16}c^6 +
         b^{20}c^5 + (b^{22} + b^{16})c^4 + b^{22}c^3 + (b^{24} +
        b^{20})c^2 + b^{24}.
\end{eqnarray*}
We first consider the case $E_4=0$. Let $e=b^2+c^2+w$. Substituting
$e$ in $E_4$ by $b^2+c^2+w$ to get
$$w^2+(c+1)w+b^2(c+1)^2=0.$$
Dividing both sides by $(c+1)^2$, we get
$(w/(c+1))^2+w/(c+1)+b^2=0$, which by Lemma \ref{lemm3.4} implies
that $\Tr(b)=0$. Thus by the comments immediately before the
statement of Theorem~\ref{lucas}, there exists a linear substitution
to eliminate $b$. Hence we may assume that $f(x)=x^6+x^5+ c x^3 +d
x^2 +e x$ (i.e., $b=0$). It follows that $$E_4=(e+c^2)(e + c^2 + c +
1).$$ Now $E_4=0$ leads to the following two cases.

\noindent{\bf Case 1.} $e+c^2=0$.  Then $E_1=0$ becomes $cd^2=0$, or
$d=0$, since $c\ne0$.  Then
$$f(x)=x^6+x^5+cx^3+c^2x=x(x^5+x^4+cx^2+c^2).$$
In the degree 5 factor of $f(x)$, substituting $x$ by $y+1$, we
obtain $y^5+cy^2+y+c^2+c$, which has exactly one root in $\Ff_{2^t}$
by Lemma~\ref{oneroot}. Therefore, $f(x)$ has two roots in
$\Ff_{2^t}$, contradicting the assumption that $f$ is a \PP.

\vspace{1 ex}

\noindent{\bf Case 2}: $e+c^2+c+1=0$.  Then $E_1=0$ becomes $d=c+1$.
Thus $f(x)$ must have the form
$$f(x)=x^6+x^5+c x^3+(c+1)x^2+(c^2+c+1)x.$$
Let us consider the following polynomial $g(x)$ that is equivalent
to $f(x)$. Set
$$g(x)=x^6+x^5+cx^3+(1+c)x^2+(c^2+c+1)x+c^2+c.$$
Then $g(x)=(x+1)(x^5+cx^2+x+c^2+c)$. The second factor has exactly
one root in $\Ff_{2^t}$ by Lemma~\ref{oneroot}. Thus $g(x)$ has two
roots in $\Ff_{2^t}$. Hence $g(x)$ cannot be a \PP. But $g(x)$ is
equivalent to $f(x)$ and $f(x)$ is assumed to be a \PP; we have
reached a contradiction.

Therefore we conclude that $E_4\neq 0$. We must have $E_6=E_7=0.$
Viewing $E_7$ and $E_6$ as polynomials in $e$, by long division, the
remainder of $E_7$ upon division by $E_6$ is
\begin{eqnarray*}
E_{10}:&=&{c}^{16}{b}^{8}+{c}^{16}{e}^{4}+{c}^{18}e+{c}^{19}e+{c}^{20}{b}^{4}+{c
}^{20}{b}^{2}
+{c}^{17}e+{c}^{8}{b}^{4}\\
&+&{b}^{2}{c}^{9}+{c}^{10}{b}^{2}+
{c}^{8}{b}^{8}
+e{c}^{11}+{c}^{9}e+{b}^{2}{c}^{18}+{c}^{10}e+{c}^{19}{b
}^{2}+{c}^{16}e\\
&+&{b}^{2}{c}^{12}+{c}^{12}{b}^{4}+{c}^{8}e+{c}^{17}{b}^{
2}+{c}^{8}{e}^{4}
+{c}^{11}{b}^{2}+{c}^{16}{b}^{4}\\
&+&{c}^{10}+{c}^{8}+{c}
^{12}+{c}^{11}+{c}^{18}+{c}^{16}+{c}^{13}+{c}^{21}+{c}^{19}+{c}^{20}.
\end{eqnarray*}
Again, $E_{10}$ can be factored into
$$E_{10}=c^8(c+1)^8\cdot E_9,$$ where
\begin{eqnarray*}
E_9&=&e^4 + (c^3 + c^2 + c + 1)e + c^5 + (b^4 + b^2 + 1)c^4\\
& +& (b^2 + 1)c^3 +
        (b^2 + 1)c^2 + b^2c + b^8 + b^4 + 1.
\end{eqnarray*}
Since $c\neq 0$ or 1, we must have $E_9=0$, using which $E_6$ can be
reduced to $c^6+c^4+c^2+1=(c+1)^6$. Since $c\ne 1$ we see that
$E_6\neq 0$, a contradiction. The proof is complete.
\end{proof}

\section{Permutation polynomials of degree $7$ over finite fields of characteristic $2$}
In this section we determine all permutation polynomials of degree
$7$ over finite fields of characteristic $2$. The first lemma
addresses the trivial case, in which $2^t\equiv 1$ (mod 7).

\begin{Lemma}
Let $t>0$ be such that $2^t\equiv 1\  (\mathrm{mod}\ 7)$.  Then
there are no permutation polynomials over $\Ff_{2^t}$ of degree $7$.
\end{Lemma}
\begin{proof} Let $f(x)$ be a monic polynomial over $\Ff_{2^t}$ of degree $7$.
Note that when $2^t=7m+1$ we have $[x^{7m }]f(x)^m=1\ne0$. Hence by
the Hermite-Dickson criterion, $f(x)$ can not be a permutation
polynomial.
\end{proof}

\begin{Lemma}
Let $t>7$ be such that $2^t\equiv 2\ (\mathrm{mod}\ 7)$.  Then every
permutation polynomial for $\Ff_{2^t}$ of degree $7$ is equivalent
to either $x^7+x^5+x$ or $x^7$.
\end{Lemma}
\begin{proof}
We proceed as in Section~\ref{degree6}.
We set $q=7m+2$, with
$$m=2^1+2^4+\cdots+2^{t-3}.$$
Suppose that $f(x)$ is a monic degree 7 \PP\ over $\Ff_q$. By
Lemma~\ref{shift}, we may normalize $f(x)$ such that
$$f(x)=x^7+bx^5+cx^4+dx^3+ex^2+fx\in\Ff_q[x].$$
Then we have the following conditions:
\begin{eqnarray}
[x^{7m+1}]f(x)^{m+1}
&=&c^2+b^3+f=0;\label{G1}\\
\ [x^{7m+1}]f(x)^{m+3}
&=& e^4+d^5=0;\label{G3}\\
\ [x^{7m+1}]f(x)^{m+11}
&=&(c^{24}+b^{16}e^8)d+(d^{16}+c^{16}b^8+b^{16}d^8)(c^4+b^4d)
\nonumber\\&&
+\ b^{16}c^8(e^4+d^5)+(c^{16}+b^{24}+f^8)f^4d;\label{G11}\\
\ [x^{7m+1}]f(x)^{m+13}
&=& (d^{16}c^8+c^{16}e^8)b+
\Big(e^{16}+d^{16}b^8+c^{16}d^8+b^{16}f^8\Big)
\nonumber\\
&&\ \cdot\Big(b^5+e^2+d^2b+c^2d+b^2f\Big)
+\Big(c^{24}+b^{16}e^8\Big)\cdot
\nonumber\\
&&\Big(d^4b+c^4(c^2+b^3+f)+\nonumber\\&&
b^4(e^2+d^2b+c^2d+b^2f)+f^3\Big)+
\nonumber\\&&
(d^{16}+c^{16}b^8+b^{16}d^8)\Big(f^4b+e^4(c^2+b^3+f)+
\nonumber\\
&&d^4(e^2+d^2b+c^2d+b^2f)+c^4(f^2b+e^2d+d^2f)\nonumber\\&&
+\ b^2f^3\Big)+b^{16}c^8\Big(f^4(e^2+d^2b+c^2d+b^2f)+
\nonumber\\&&
e^4(f^2b+e^2d+d^2f)+d^4f^3\Big)+\nonumber\\
&&(c^{16}+b^{24}+f^8)f^7;\label{G13}\\
\ [x^{7m+1}]f(x)^{m+19}
&=&d^{33}=0.\label{G19}
\end{eqnarray}
Combining (\ref{G3}) and (\ref{G19}), we have $d=e=0$.  By the
comments immediately before the statement of Theorem~\ref{lucas} in
Section 2, we may assume that $b\in\{0,1\}$. First we assume that
$b=1$. Then (\ref{G11}) reduces to $c^{20}=0$ and (\ref{G1}) gives
us $f=1$. The \PP\ is
$$x^7+x^5+x.$$
Next we assume that $b=0$.  Then (\ref{G13}) reduces to
$c^{24}f^3=0$ (using (\ref{G1}) to eliminate the last term and one
other term), which combined with (\ref{G1}) gives us $c=f=0$. The
\PP\ is $x^7$. The proof is complete.
\end{proof}

\begin{Lemma}
Let $t>5$ be such that $2^t\equiv 4\ (\mathrm{mod}\ 7)$. Then every
permutation polynomial over $\Ff_{2^t}$ of degree $7$ is equivalent
either $x^7+x^5+x$ or $x^7$.
\end{Lemma}

\begin{proof}
We  set $q=7m+4$ with
$$m=2^2+2^5+\cdots+2^{t-3}.$$
Suppose that $f(x)$ is a monic degree 7 \PP\ over $\Ff_q$. By
Lemma~\ref{shift}, we may normalize $f(x)$ such that
$$f(x)=x^7+bx^5+cx^4+dx^3+ex^2+fx\in\Ff_q[x].$$
This time we use the following relations:
\begin{eqnarray}
[x^{7m+3}]f(x)^{m+1}
&=&d=0;\label{H1}\\
\ [x^{7m+3}]f(x)^{m+3}
&=&d^4b+c^4(c^2+b^3+f)+
\nonumber\\&&
b^4(e^2+d^2b+c^2d+b^2f)+ f^3=0;\label{H3}\\
\ [x^{7m+3}]f(x)^{m+9}
&=&f^8c^4+e^{12}+(f^8b^4+e^8d^4+d^8f^4)d=0;\label{H9}\\
\ [x^{7m+3}]f(x)^{m+15}
&=&(c^{32}+b^{48}+f^{32})(c^2+b^3+f)=0;\label{H15}\\
\ [x^{7m+3}]f(x)^{m+19}
&=&(d^{32}+c^{32}b^{16}+b^{32}d^{16})b+b^{32}c^{16}\Big(d^4b+
\nonumber\\&&c^4(c^2+b^3+f)+b^4(e^2+d^2b+c^2d+b^2f)+f^3\Big)
\nonumber\\&&
+\ \Big(c^{32}+b^{48}+f^{16}\Big)\Big(f^4(e^2+d^2b+c^2d+b^2f)+
\nonumber\\&&
e^4(f^2b+e^2d+d^2f)+d^4f^3\Big)=0;\label{H19}
\end{eqnarray}
From (\ref{H1}) and (\ref{H15}) we get $d=0$ and $f=c^2+b^3$.  By
the comments immediately before the statement of Theorem~\ref{lucas}
in Section 2, we may assume that $b\in\{0,1\}$.  First assume that
$b=1$. Substituting into (\ref{H3}) and (\ref{H19}) we get
$e^2+f+f^3=0$ and $c^{32}+e^2+f+f^3=0$, or $c=0$.  We have $f=1$,
and (\ref{H9}) reduces to $e^{12}=0$.  The \PP\ is $x^7+x^5+x$. Now
assume $b=0$. Then (\ref{H3}) reduces to $f^3=0$, and we have
$b=c=d=f=0$. Substituting into (\ref{H9}) gives $e^{12}$ and the
\PP\ is $x^7$. The proof is now complete.
\end{proof}

Combining the lemmas and computer results for $\Ff_{16},\ \Ff_{32},$ and
$\Ff_{128}$, we have:
\begin{Theorem}
Let $2^t\ge8$.  If $t\equiv0\ (\mathrm{mod}\ 3)$, then there are no
permutation polynomials over $\Ff_{2^t}$ of degree $7$.  Otherwise,
every permutation polynomial over $\Ff_{2^t}$ of degree $7$ is
equivalent to $x^7+x^5+x$, or to $x^7$, or $t=4$, and the polynomial
is equivalent to one of the following:
\begin{eqnarray*}
&x^7+a^3x^4+a^6x&\\
&x^7+x^5+x^47\\
&x^7+x^5+ax^4+a^{14}x^3+a^{12}x^2+a^8x&\\
&x^7+x^5+a^5x^4+a^2x^3+a^{12}x^2+a^5x&\\
&x^7+x^5+a^7x^4+a^5x^2+a^3x,&
\end{eqnarray*}
where $a$ is a root of $x^4+x+1\in\Ff_2[x]$ in some extension field
of $\Ff_2$.
\end{Theorem}

\begin{rem}
The methods that we used for classifying \PP s of $\Ff_{2^t}$ of
degree 6 or 7 will not work when the degree of the \PP\ is a power
of the characteristic. In other cases, for fixed degree and fixed
characteristic, the methods can be expected to work, although the
number of terms to deal with is likely to increase rapidly with the
characteristic.
\end{rem}





\end{document}